\documentclass[12pt, leqno]{amsart}
\usepackage{amsmath,amsfonts,amsthm,latexsym,amssymb,mathrsfs}
\usepackage{amsmath}
\usepackage{amssymb}
\usepackage{graphicx}
\usepackage{amsthm}
\usepackage{amsfonts}
\usepackage{amscd}

\def\d#1{{#1\kern-0.4em\char"16\kern-0.1em}}
\def\D#1{{\raise0.2ex\hbox{-}\kern-0.4em#1}}
\def\dzn{,\kern-0.1em,}
\newtheorem{defi}{Definition}[section]
\newtheorem{teor}[defi]{Theorem}

\newtheorem{prop}[defi]{Proposition}

\newtheorem{rem}[defi]{Remark}
\newtheorem{corollary}[defi]{Corollary}

\def\A{{\mathcal{A}}}
\def\W{{\mathcal{W}}}
\def\N{{\mathcal{N}}}
\def\D{{\mathcal{D}}}
\def\KD{{\mathcal{KD}}}
\def\B{{\mathcal{B}}}
\def\C{{\mathcal{C}}}
\def\F{{\mathcal{F}}}
\def\NN{\mathbb{N}}
\def\CC{\mathbb{C}}

\def\dj{d\kern-0.4em\char"16\kern-0.1em}
\def\Dj{\mbox{\raise0.3ex\hbox{-}\kern-0.4em D}}

\def\R{{\mathcal{R}}}
\def\N{{\mathcal{N}}}
\def\D{{\mathcal{D}}}
\def\B{{\mathcal{B}}}
\def\C{{\mathcal{C}}}
\def\F{{\mathcal{F}}}
\def\K{{\mathcal{K}}}
\def\L{{\mathcal{L}}}
\def\X{{\mathcal{X}}}
\def\H{{\mathcal{H}}}

\def\T{{\mathcal{T}}}
\def\G{{\mathcal{G}}}

\def\NN{\mathbb{N}}
\def\CC{\mathbb{C}}

\def\dj{d\kern-0.4em\char"16\kern-0.1em}
\def\Dj{\mbox{\raise0.3ex\hbox{-}\kern-0.4em D}}

\def\Poly{{\rm Poly}}
\def\acc{{\rm acc}}
\def\iso{{\rm iso}}
\def\left{{\rm left}}
\def\right{{\rm right}}
\def\ind{{\rm ind}}

\begin{document}
\date{}
\title[GENERALIZED B-FREDHOLM  ELEMENTS]{GENERALIZED B-FREDHOLM BANACH \\ ALGEBRA ELEMENTS}

\author[M. D. CVETKOVI\'C]{MILO\v S D. CVETKOVI\'C}
\address{UNIVERSITY OF NI\v S\\ FACULTY O SCIENCES AND MATHEMATICS\\ VI\v SEGRADSKA 33, P.O. BOX 224, 18000 NI\v S, SERBIA}
\email{milosCvetkovic83@gmail.com}

\author[E. BOASSO]{ENRICO BOASSO}
\email{enrico\_odisseo@yahoo.it}

\author[S. \v C. \v ZIVKOVI\'C-ZLATANOVI\'C]{SNE\v ZANA \v C. \v ZIVKOVI\'C-ZLATANOVI\'C}
\address{UNIVERSITY OF NI\v S\\ FACULTY O SCIENCES AND MATHEMATICS\\ VI\v SEGRADSKA 33, P.O. BOX 224, 18000 NI\v S, SERBIA}
\email{mladvlad@open.telekom.rs}

\begin{abstract}\noindent Given a (not necessarily continuous) homomorphism between Banach algebras $\T\colon\A\to\B$,
an element $a\in\A$ will be said to be B-Fredholm (respectively
generalized B-Fredholm) relative to $\T$, if $\T(a)\in \B$ is Drazin
invertible (respectively Koliha-Drazin invertible). In this article,
the aforementioned elements will be characterized and their main
properties will be studied. In addition, perturbation properties
will be also considered.
\end{abstract}


\maketitle
\section{Introduction}

The Atkinson theorem states that necessary and sufficient for a
Banach space operator to be Fredholm is that its coset in the Calkin
algebra is invertible. This well known result led to the
introduction of a Fredholm theory relative to a Banach algebra
homomorphism, see \cite{Harte}. This theory has been developed by
many authors, which have also studied other classes of objects such
that Weyl, Browder and  Riesz Banach algebra elements
relative to a (not necessarily continuous) homomorphism, see for
example \cite{Harte, H3, H2, H4, MR1, GR, MR, Dragan, H5, Bak,
ZH, MMH, ZDH, ZDH2}. \par

\indent Recall that the theory of Fredholm operators was
generalized. In fact, the notion of B-Fredholm operator was
introduced and studied, see section $2$ or  \cite{Ber3, Mu2}.  In
addition,  given $\X$ a Banach space and $T\in \L (\X)$ a bounded
and linear map, according to \cite[Theorem 3.4]{Ber1}, $T$ is
B-Fredholm if and only if $\tilde{\pi}(T) \in \L (\X )/ \F (\X )$
is Drazin invertible, where $\F (\X)$ denotes the ideal of finite
rank operators defined on $\X$ and $\tilde{\pi}: \L (\X ) \to \L (\X
)/ \F (\X )$ is the quotient homomorphism. This Atkinson-type
theorem for B-Fredholm operators leads to the following definition.
Let $\A$ and $\B$ be two complex unital Banach algebras and consider
a (not necessarily continuous) homomorphism $\T\colon \A\to\B$. The
element $a\in\A$ will be said to be B-Fredholm (respectively
generalized B-Fredholm) relative to $\T$, if $\T(a)\in \B$ is Drazin
invertible (respectively Koliha-Drazin invertible).\par

\indent The main objective of this article is to study (generalized) B-Fredholm Banach algebra
elements relative to a (not necessarily continuous) homomorphism. In section 3, after having recalled some
preliminary facts in section 2, the aforementioned elements will be characterized and
studied. What is more, B-Weyl and B-Browder elements will be also considered. On the other hand,
in section 4 perturbation properties of B-Fredholm elements will be studied and in section 5 perturbations of
(generalized) B-Fredholm elements with equal spectral idempotents with respect to a (not necessarily continuous) homomorphism
will be considered.\par

\section{Preliminary Definitions and Facts}

From now on $\A$ will denote a complex unital Banach algebra with
identity $1$. Let $\A^{-1}$, $\A_{\left}^{-1}$, $\A_{\right}^{-1}$,
$\A^\bullet$ and $\A^{nil}$ denote the set of all invertible
elements in $\A$, the set of all left  invertible elements in
$\A$, the set of all right  invertible elements in
$\A$, the set of all idempotents in $\A$ and the set of all
nilpotent elements in $\A$, respectively. Given $a\in \A$, $\sigma
(a)$ and $\iso \, \sigma (a)$ will stand for the spectrum and the
set of isolated spectral points of $a\in \A$, respectively. Recall,
if $K\subseteq \mathbb{C}$, then $\acc \, K$ is the set of limit
points of $K$ and $\iso \, K=K\setminus \acc \, K$. \par

\indent An element $a \in \A$ is said to be Drazin invertible, if there exists a necessarily unique
$b\in \A$ and some  $k\in\NN$ such that
$$
bab=b, \; \; \; \; ab=ba, \; \; \; \;  a^{k}ba=a^k.
$$

\noindent If the Drazin inverse of $a$ exists, then it will be
denoted by $a^d$. In addition, the  \it index of $a$, \rm  which
will be denoted by  $\ind \, (a)$, is the least non-negative integer
$k$ for which the above equations hold. When $\ind \, (a)=1$, $a$
will be said to be \it group invertible\rm, and in this case its
Drazin inverse will be referred as the group inverse of $a$;
moreover, it will be denoted by $a^\sharp$. The set of all Drazin
invertible (respectively group invertible) elements of $\A$ will be
denoted by $\A^{\D}$ (respectively $A^\sharp$); see \cite{Dr, K,
RS}.\par

\indent Recall that $a\in \A$ is said to be \it generalized Drazin \rm  or \it Koliha-Drazin invertible\rm,
if there exists $b\in A$ such that
$$
bab=b, \; \; \; \;  ab=ba,  \; \; \; \; aba=a+w,
$$

\noindent where $w\in\A^{qnil}$ (the set of all quasinilpotent elements of $\A$). The Koliha-Drazin inverse is unique, if it exists,
and it will be denoted by $a^D$; see \cite{Koliha,Lu, R2, R3}.

\indent Note that if $a\in \A$ is Drazin invertible, then $a$ is
generalized Drazin invertible. In fact, if $k=\ind \, (a)$, then
$(aba-a)^k=0$. In addition, according to \cite[Proposition 1.5]{Lu}
(see also \cite[Theorem 4.2]{Koliha}), necessary and sufficient for
$a\in \A$ to be Koliha-Drazin invertible but not invertible is that
$0\in \iso \, \sigma (a)$. The set of all Koliha-Drazin invertible
elements of $\A$ will be denoted by $\A^{\KD}$.

Recall that given a Banach algebra $\A$, a subset $\R\subset \A$ is said to be a \it regularity, \rm if the
following two conditions are satisfied (see for example \cite{KM, Mu}):\par
\medskip
(i) Given $a \in \A$ and $n \in \NN$, $a \in \R$ if
and only if $a^n \in \R$.\par
(ii) If $a,b,c,d \in \A$ are
mutually commuting elements satisfying $ac+bd=1$, then necessary and sufficient for $ab\in \R$ is that $a,b \in \R$.\par
\medskip
Given a regularity $\R\subset \A$, it is possible to define the
spectrum of $a\in\A$ corresponding to $\R$ as $\sigma_\R (a)=\{
\lambda\in\CC\colon a-\lambda\notin \R\}$ ($a\in\A$). This spectrum
satisfies the spectral mapping theorem for every $a\in \A$ and every
analytic function defined on a neighbourhood of $\sigma (a)$ which
is non-constant on each component of its domain of definition; see
\cite[Theorem 1.4]{KM}. In particular, according to \cite[Theorem 2.3]{Ber1} and
\cite[Theorem 1.2]{Lu}, the sets $\A^{\D}=\{a\in \A\colon a\hbox{ is Drazin
invertible}\}$ and $\A^{\KD}=\{a\in \A\colon a\hbox{ is
Koliha-Drazin invertible}\}$ are regularities, respectively. The corresponding
Drazin and Koliha-Drazin spectra will be denoted by $\sigma_\D (a)$
and $\sigma_\KD (a)$, respectively.\par
\bigskip

\indent Let $\X$ be a Banach space and denote by $\L(\X)$ the
algebra of all bounded and linear maps defined on and with values in
$\X$. If $T\in \L(\X)$, then $T^{-1}(0)$ and $R(T)$ will stand for
the null space and the range of $T$ respectively. Note that $I\in
\L(\X)$ will denote the identity map of $\X$. In addition,
$\K(\X)\subset \L(\X)$ will stand for the closed ideal of compact
operators. Consider $\mathcal{C}(\X)$ the Calkin algebra over $\X$,
i.e.,  the quotient algebra $\mathcal{C}(\X)=\L(\X) / \K(\X)$.
Recall that $\mathcal{C}(\X)$ is itself a Banach algebra with the
quotient norm. Let
$$
\pi\colon \L(\X)\to \mathcal{C}(\X)
$$
denote the quotient map.\par Recall that  $T\in \L(\X)$ is said to
be a \it Fredholm operator, \rm if $T^{-1}(0)$ and $\X/R(T)$ are
finite dimensional. Denote by $\Phi (\X)$ the set of all Fredholm
operators defined on $\X$. It is well known that $\Phi (\X)$ is a
multiplicative open semigroup in $\L(\X)$ and that
$$
\Phi (\X)= \pi^{-1}(\mathcal{C}(\X)^{-1}).
$$

\indent Atkinson's theorem motivated the Fredholm theory relative to
homomorphism between two Banach algebras as well as the introduction
of Browder and Weyl elements relative to a such homomorphism. This
theory was introduced in \cite{Harte}. Next follow the main notions.

\begin{defi}\label{def1}\rm Let $\A$ and $\B$ be two unital Banach algebras and consider  a (not necessarily continuous) homomorphism
$\T:  \A \to \B$. An element $a\in \A$ will be said to be\par
\noindent {\rm (i)}  \it Fredholm, \rm if $\T(a)$ is invertible in $\B$;\par
\noindent {\rm (ii)} \it Weyl, \rm if there exist $b,c \in \A$,
$b\in\A^{-1}$ and $c\in \T^{-1}(0)$, such that $a=b+c$; \par
\noindent {\rm (iii)} \it Browder, \rm if there exist $b,c \in \A$,
$b\in\A^{-1}$, $c\in \T^{-1}(0)$ and $bc=cb$ ,  such that $a=b+c$.
\end{defi}

\indent It is worth noting that the definitions of Weyl and Browder
elements are also motivated from the Banach space operator case (see
\cite[p.431]{Harte}). The sets of Fredholm, Weyl and Browder
elements relative  to the homomorphism $\T\colon \A\to\B$ will be
denoted by $\F_{\T}(\A )$, $\W_{\T}(\A )$ and $\B_{\T}(\A )$,
respectively. Naturally, these sets lead to the introduction of the
corresponding spectra.

\begin{defi}\label{def2}\rm Let $\A$ and $\B$ be two unital Banach algebras and consider  a (not necessarily continuous) homomorphism
$\T:  \A \to \B$. Given $a\in \A$, the \it Fredholm spectrum, \rm the \it Weyl spectrum \rm  and the \it Browder spectrum \rm of $a$
relative  to the homomorphism $\T\colon \A\to\B$ is\par
\noindent {\rm (i)}  $\sigma_{\F_{\T}}(a)=\{\lambda \in \CC : a-\lambda \not\in \F_{\T}(\A )\}=\sigma (\T (a))$,\par
\noindent {\rm (ii)} $\sigma_{\W_{\T}}(a)=\{\lambda \in \CC : a-\lambda \not\in \W_{\T}(\A )\}$,\par
\noindent {\rm (iii)}  $\sigma_{\B_{\T}}(a)=\{\lambda \in \CC : a-\lambda \not\in \B_{\T}(\A )\}$,\par
\noindent respectively.
\end{defi}

\indent It is clear that  $\B_{\T}(\A ) \subset \W_{\T}(\A ) \subset
\F_{\T}(\A )$ and that $\sigma_{\F_{\T}}(a) \subset
\sigma_{\W_{\T}}(a) \subset \sigma_{\B_{\T}}(a) \subset \sigma(a)$.
Also it is known that the sets $\sigma_{\F_{\T}}(a),
\sigma_{\W_{\T}}(a)$ and $\sigma_{\B_{\T}}(a)$ are non-empty and
compact. To learn the main properties of these objects, see for
example \cite{Harte, H3, H2, H4, MR1, GR, MR, Dragan, H5, Bak,
ZH, MMH, ZDH, ZDH2}.
\bigskip

\indent Recall that an operator $T\in \L (\X )$ is said to be \it
B-Fredholm, \rm if there is $n\in\NN$ such that $R(T^n)$ is closed
and $T\mid_{R(T^n)} \in\L (R(T^n))$ is Fredholm, see  \cite{Ber3,
Mu2}. Denote by $\mathcal{BF}(\X )$ the class of B-Fredholm operators
defined on the Banach space $\X$.  According to \cite[Theorem
3.4]{Ber1}, $T \in \mathcal{BF}(\X )$  if and only if $\tilde{\pi}(T)
\in (\L (\X )/ \F (\X ))^\D$, where $\F (\X)$ denotes the ideal of
finite rank operators defined on $\X$ and $\tilde{\pi}: \L (\X ) \to
\L (\X )/ \F (\X )$ is the quotient homomorphism. In addition, given
a Hilbert space $\H$, according to \cite[Theorem 3.12(ii)]{Boasso1},
$\pi (\mathcal{BF}(\H))=\mathcal{C}(\H)^D$. Moreover, B-Weyl operators
were introduced in \cite{Ber2}  and according to  \cite[Corollary
4.4]{Ber2}, $T \in \L (\X )$ is a B-Weyl if and only if $T=S+F$,
where $S\in \L (\X )^\D$ and $F\in \F (\X )$.\par

\indent On the other hand,  the classes of Riesz-Fredholm and power
compact-Fredholm operators on the Banach space $\X$ were introduced
in \cite{Boasso1} and they will be denoted by $\mathcal{RF}(\X )$
and  $\mathcal{PKF}(\X )$, respectively. According to \cite[Theorem
3.11]{Boasso1}, $\pi (\mathcal{RF}(\X))=\C(\X)^{\KD}$ and $\pi
(\mathcal{PKF}(\X))=\C(\X)^{\D}$.\par \indent These observations led
to the following definition.

\begin{defi}\label{def3}\rm Let $\A$ and $\B$ be two unital Banach algebras and consider a (not necessarily continuous)  homomorphism $\T\colon \A\to\B$.
An element $a\in \A$ is said to be \par \noindent {\rm (i)}
\it B-Fredholm \rm (respectively \it B-Fredholm of degree $k$\rm), if $\T(a)\in
\B^\D$ (respectively $\T(a)\in \B^\D$ and $\ind \, T(a)=k$);\par
\noindent {\rm (ii)} \it B-Weyl \rm (respectively \it B-Weyl of degree $k$\rm), if
there exist $b,c \in  \A$, $b\in \A^\D$ (respectively $b\in\A^\D$
and $\ind \, b=k$) and $c\in \T^{-1}(0)$, such that $a=b+c$;\par
\noindent {\rm (iii)} \it B-Browder \rm (\it B-Browder of degree $k$\rm) if there
exist  $b,c \in \A$, $bc=cb$, $b\in \A^\D$ (respectively $b\in\A^\D$
and $\ind \, b=k$) and $c\in \T^{-1}(0)$, such that $a=b+c$;\par
\noindent {\rm (iv)} \it generalized B-Fredholm\rm, if $\T(a)\in
\B^{\KD}$;\par \noindent {\rm (v)} \it generalized B-Weyl\rm, if there
exist $b,c \in  \A$, $b\in \A^\KD$ and $c\in \T^{-1}(0)$, such that
$a=b+c$;\par  \noindent {\rm (vi)} \it generalized B-Browder\rm, if there
exist $b,c \in \A$, $bc=cb$, $b\in \A^\KD$ and $c\in \T^{-1}(0)$,
such that $a=b+c.$
\end{defi}

\indent The set of B-Fredholm (respectively B-Fredholm of degree
$k$, B-Weyl, B-Weyl of degree $k$, B-Browder, B-Browder of degree
$k$, generalized B-Fredholm, generalized B-Weyl and generalized
B-Browder) elements of the unital Banach algebra $\A$ relative to
the homomorphism $\T\colon \A\to \B$ will be denoted by
$\B\F_{\T}(\A)$ (respectively $\B\F^{(k)}_{\T}(\A)$,
$\B\W_{\T}(\A)$, $\B\W^{(k)}_{\T}(\A)$, $\B\B_{\T}(A)$,
$\B\B^{(k)}_{\T}(\A)$, $\mathcal{GBF}_{\T}(\A)$,
$\mathcal{GBW}_{\T}(\A)$ and $\mathcal{GBB}_{\T}(\A)$). Next the
corresponding spectra will be introduced.

\begin{defi}\label{def4}\rm Let $\A$ and $\B$ be two unital Banach algebras and consider  a (not necessarily continuous) homomorphism
$\T:  \A \to \B$. Given $a\in \A$, the \it B-Fredholm spectrum\rm, the
\it B-Weyl spectrum\rm, the \it B-Browder spectrum\rm, the \it generalized B-Fredholm
spectrum\rm, the \it generalized B-Weyl spectrum \rm and the \it generalized
B-Browder spectrum \rm of $a$ relative to the homomorphism $\T\colon
\A\to\B$ is\par \noindent {\rm (i)} $\sigma_{\B\F_{\T}}(a)=\{\lambda
\in \CC : a-\lambda \not\in \B\F_{\T}(\A )\}=\sigma_{\D}(\T
(a))$,\par \noindent {\rm (ii)} $\sigma_{\B\W_{\T}}(a)=\{\lambda \in
\CC : a-\lambda \not\in \B\W_{\T}(\A )\}$,\par \noindent {\rm (iii)}
$\sigma_{\B\B_{\T}}(a)=\{\lambda \in \CC : a-\lambda \not\in
\B\B_{\T}(\A )\}$,\par \noindent {\rm (iv)}
$\sigma_{\mathcal{GBF}_{\T}}(a)=\{\lambda \in \CC : a-\lambda
\not\in \mathcal{GBF}_{\T}(\A )\}=\sigma_{\KD}(\T (a))$,\par
\noindent {\rm (v)} $\sigma_{\mathcal{GBW}_{\T}}(a)=\{\lambda \in
\CC : a-\lambda \not\in \mathcal{GBW}_{\T}(\A )\}$,\par \noindent
{\rm (vi)} $\sigma_{\mathcal{GBB}_{\T}}(a)=\{\lambda \in \CC :
a-\lambda \not\in \mathcal{GBB}_{\T}(\A )\}$,\par \noindent
respectively.
\end{defi}

\begin{rem}\label{rem5}\rm Let $\A$ and $\B$ be two unital Banach algebras and consider  a (not necessarily continuous) homomorphism
$\T:  \A \to \B$ and $a\in A$. Then, it is not difficult to prove
the following statements:\par \noindent {\rm (i)} $\F_{\T}(\A)
\subseteq \B\F_{\T}(\A)\subseteq \mathcal{GBF}_{\T}(\A)$.\par

\noindent {\rm (ii)} $\B\F_{\T}(\A)+\T^{-1}(0)= \B\F_{\T}(\A)$.\par

\noindent {\rm (iii)} $\W_{\T}(\A)\subseteq
\B\W_{\T}(\A)=\A^\D+\T^{-1}(0)\subseteq \A^\KD+\T^{-1}(0)=\G\B\W_{\T}(\A)$ and $\B_{\T}(\A) \subseteq
\B\B_{\T}(\A)\subseteq \G\B\B_{\T}(\A)$.\par

\noindent {\rm (iv)} $\A^\D\subseteq
\B\B_{\T}(\A)\subseteq \B\W_{\T}(\A)\subseteq\B\F_{\T}(\A)$ and
$\A^\KD\subseteq \G\B\B_{\T}(\A)\subseteq \G\B\W_{\T}(\A)\subseteq\G\B\F_{\T}(\A)$.\par

\noindent {\rm (v)}
$\mathcal{GBF}_{\T}(\A)+\T^{-1}(0)=\mathcal{GBF}_{\T}(\A)$.\par

\noindent {\rm (vi)} $\A^\D\subseteq \A^{\KD}\subseteq
\mathcal{GBF}_{\T}(\A)$.\par
\noindent {\rm (vii)}
$\sigma_{\mathcal{GBF}_{\T}}(a)\subseteq
\sigma_{\mathcal{\B\F}_{\T}}(a)\subseteq
\sigma_{\mathcal{\F}_{\T}}(a)$.\par
\noindent {\rm (viii)}
$\sigma_{\G\B\W_{\T}}(a)\subseteq\sigma_{\B\W_{\T}}(a)\subseteq \sigma_{\W_{\T}}(a)$ and
$\sigma_{\G\B\B_{\T}}(a)\subseteq\sigma_{\B\B_{\T}}(a)\subseteq \sigma_{\B_{\T}}(a)$.\par
 \noindent
{\rm (ix)} $\sigma_{\mathcal{BF}_{\T}}(a)\subseteq
\sigma_{\B\W_{\T}}(a)\subseteq
\sigma_{\B\B_{\T}}(a)\subseteq\sigma_\D (a)$ and
$\sigma_{\mathcal{GBF}_{\T}}(a)\subseteq
\sigma_{\G\B\W_{\T}}(a)\subseteq
\sigma_{\G\B\B_{\T}}(a)\subseteq\sigma_\KD (a)$. \par
\noindent {\rm
(x)} $\sigma_{\mathcal{GBF}_{\T}}(a)\subseteq
\sigma_{\KD}(a)\subseteq \sigma_\D(a)$.
\end{rem}

Finally, given $S \subset \A$, $ \Poly^{-1}(S)=\{a\in\A\colon
\exists \hbox{ } p\in \CC [X],  p\neq 0 \hbox{ and } p(a) \in
S\}$, where $\CC [X]$ is the  algebra of complex polynomials. In
particular, $\Poly^{-1}(\{$ $0\})$ is the set of algebraic elements
of $\A$.
\bigskip

\section{Generalized B-Fredholm elements}

In this section, the basic properties of the
objects studied in this work will be considered. To this end,
however, first it is necessary to recall some facts. \par

\begin{rem}\label{rem2.1}\rm Let $\A$ be a  unital Banach algebra and consider $a\in \A$. Recall that according to
\cite[Theorem 12(iv)]{Bo}, $\sigma_{\D}(a)= I(a)\cup \, \acc \, \sigma (a)$, where $I(a)=$ iso $\sigma (a)\setminus \Pi(a)$ and $\Pi (a)\subseteq \iso \, \sigma
(a)$ denotes the set of poles of $a\in \A$ (\cite[Remark 10]{Bo}). In
particular $\acc \, \sigma_{\D}(a)= \acc \, \sigma
(a)=\sigma_{\KD}(a)$ (\cite[Theorem 12(iii)]{Bo} and \cite[Proposition 1.5(i)]{Lu}). The following statements can
be easily deduced from this fact.\par \noindent (i) Let $b\in \A$.
Since $\sigma_{\D}(ab)=\sigma_{\D}(ba)$ (\cite[Theorem 2.3]{Boasso}),
$\sigma_{\KD}(ab)=\sigma_{\KD}(ba)$.\par \noindent (ii) According to
\cite[Theorem 2.2]{Boasso}, necessary and sufficient for $\sigma_{\D}(a)$ to be
countable is that $\sigma(a)$ is countable. As a result,
$\sigma_{\KD}(a)$ is countable if and only  if $\sigma_{\D}(a)$ is
countable if and only if  $\sigma(a)$ is countable.
\end{rem}

\indent In the following theorem the main properties of the (generalized) B-Fredholm spectrum will be studied.\par

\begin{teor}\label{teo2.2}Let $\A$ and $\B$ be two  unital Banach algebras and consider a (not necessarily continuous)
homomorphism $\T\colon \A\to\B$. If $a\in \A$, then the following statements hold.\par
\noindent {\rm (i)}  $\B\F_{\T}(\A)$ and $\mathcal{GBF}_{\T}(\A)$ are regularities.\par
\noindent {\rm (ii)} If $f\colon U\to \CC$ is an analytic function defined on a neighbourhood
of $\sigma (a)$ which is non-constant on each component of its domain of definition,
then
$$
\sigma_{\mathcal{BF}_{\T}}(f(a))=f(\sigma_{\mathcal{BF}_{\T}}(a)), \hbox{ and } \sigma_{\mathcal{GBF}_{\T}}(f(a))=f(\sigma_{\mathcal{GBF}_{\T}}(a)).
$$
\noindent {\rm (iii)} $\sigma_{\mathcal{\B\F}_{\T}}( a)$ and
$\sigma_{\mathcal{GBF}_{\T}}(a)$ are closed.\par \noindent {\rm
(iv)} $\sigma_{\mathcal{\B\F}_{\T}}(a)=\emptyset$ if and only if
$a\in \Poly^{-1}(\T^{-1}(0))$, equivalently, $\T (a)\in
\Poly^{-1}(0)$.\par \noindent {\rm (v)}
$\sigma_{\mathcal{GBF}_{\T}}(a)=\emptyset$ if and only if $\acc \,
\sigma_{\mathcal{F}_{\T}}(a)=\emptyset$.\par \noindent {\rm (vi)}
Given $a_1$ and $a_2\in \A$,
$\sigma_{\mathcal{\B\F}_{\T}}(a_1a_2)=\sigma_{\mathcal{\B\F}_{\T}}(a_2a_1)$
and
$\sigma_{\mathcal{GBF}_{\T}}(a_1a_2)=\sigma_{\mathcal{GBF}_{\T}}(a_2a_1)$.\par
\noindent {\rm (vii)} $\sigma_{\mathcal{\B\F}_{\T}}(a)$ is countable
if and only if $\sigma_{\mathcal{GBF}_{\T}}(a)$ is countable  if and
only if $\sigma_{\mathcal{\F}_{\T}}(a)$ is countable.
\end{teor}
\begin{proof}(i). According to \cite[Theorem 2.3]{Ber1} and \cite[Theorem 1.2]{Lu},  $\A^\D$ and $\A^{\KD}$ are regularities, respectively.
Since $\T\colon \A\to\B$ is an algebra homomorphism, $\B\F_{\T}(\A)$
and $\mathcal{GBF}_{\T}(\A)$ are regularities.\par \noindent (ii).
Apply \cite[Theorem 1.4]{KM} to $\B\F_{\T}(\A)$ and
$\mathcal{GBF}_{\T}(\A)$.\par \noindent (iii). Recall that
$\sigma_{\B\F_{\T}}(a)=\sigma_{\D}(\T (a))$ and
$\sigma_{\mathcal{GBF}_{\T}}(a)=\sigma_{\KD}(\T (a))$. Then, apply
\cite[Proposition 2.5]{Ber1} and \cite[Proposition 1.5(ii)]{Lu}.\par \noindent (iv). Since
$\sigma_{\B\F_{\T}}(a)=\sigma_{\D}(\T (a))$, this statement can be
deduced from \cite[Theorem 2.1]{Boasso}.\par \noindent (v). Recall that, according
to \cite[Proposition 1.5(i)]{Lu}, $\sigma_{\mathcal{GBF}_{\T}}(a)=\sigma_{\KD} (\T (a))=
\acc \, \sigma (\T (a))= \acc \, \sigma_{\mathcal{\F}_{\T}}(
a)$.\par \noindent (vi). Apply \cite[Theorem 2.3]{Boasso} and Remark
\ref{rem2.1}(i).\par \noindent (vii). Apply \cite[Theorem 2.2]{Boasso} and Remark
\ref{rem2.1}(ii).
\end{proof}

\indent Next, (generalized) B-Fredholm elements will be characterized. To this end, however,
some notions need to be recalled.\par

\indent Let $\A$  and $\B$ be two  unital Banach algebras and
consider a (not necessarily continuous) homomorphism $\T\colon \A\to \B$. The kernel and the
range of the homomorphism $\T$ will be denoted by $\T^{-1}(0)$ and
$R(\T)$, respectively. Let $\mathcal{R}_\T (\A)=\{a\in\A\colon \T
(a)\in\B^{qnil}\}$ be the set of Riesz elements of $\A$ relative to
the homomorphism $\T$ and $\mathcal{N}_\T (\A)=\{ a\in\A\colon
\hbox{there exists } k\in\NN \hbox{ such that } a^k\in \T^{-1}(0)
\}=\{a\in\A\colon \T (a)\in\B^{nil}\}$ be the set of $\T$-nilpotent
elements of $\A$; see \cite{Boasso1, ZDH, ZDH2}. Clearly,
$\mathcal{N}_\T (\A)\subseteq \mathcal{R}_\T (\A)$.   \par

\indent On the other hand, the homomorphism $\T\colon \A\to \B$
will be said to have the \it lifting property, \rm if given $q\in
\B^\bullet$, there is $p\in\A^\bullet$ such that $\T (p)=q$, i.e.,
$T(\A^\bullet)=\B^\bullet$, which is equivalent to the conjunction
of the following two conditions: $\T^{-1}(\B^\bullet)=\A^\bullet +
\T^{-1}(0)$ and $\B^\bullet \subset R(\T)$.
 This property does not in general hold. Some
examples that satisfy the lifting property (among others the
radical of a Banach algebra or von Neumann algebras) were
considered in \cite[Remark 3.4]{Boasso1}. In particular, if
$\B^\bullet \subset R(\T)$  and $\T$ has the \it Riesz property,
\rm i.e., if for every $z\in \T^{-1}(0)$, $\sigma (z)$ is either
finite or is a sequence converging to $0$, then $\T$ has the
lifting property, see \cite[Theorem 1]{R} and \cite[Lemma 2]{Dragan}. Therefore, if
$\T\colon\A\to \B$ is surjective and  has the  Riesz property,
then $\T$ has the lifting property. \par

\begin{teor}\label{teo2.3}Let $\A$ and $\B$ be two  unital Banach algebras and consider a (not necessarily continuous)
homomorphism $\T\colon \A\to\B$. Suppose that $\T$ has the
lifting property. Then, the following statements hold.\par \noindent
{\rm (i)} Necessary and sufficient for $a\in \mathcal{GBF}_\T (\A)$
is that there exists $p\in \A^\bullet$ such that $a+p\in \F_\T
(\A)$, $pa(1-p)$ and $(1-p)ap\in \T^{-1}(0)$ and $pap\in
\mathcal{R}_\T (\A)$.\par

\noindent {\rm (ii)} Necessary and sufficient for $a\in
\mathcal{BF}_\T (\A)$ is that there exists $p\in \A^\bullet$ such
that $a+p\in \F_\T (\A)$, $pa(1-p)$ and $(1-p)ap\in \T^{-1}(0)$ and
$pap\in \mathcal{N}_\T (\A)$.\par
\end{teor}
\begin{proof}(i). If $a\in\mathcal{GBF}_\T (\A)$, then $\T (a)\in \B^{\KD}$. In particular, according to \cite[Theorem 4.2]{Koliha},
there is $q\in\B^\bullet$ such that $q\T (a)=\T (a)q$, $\T (a)+q\in \B^{-1}$ and $\T (a)q=q\T(a)q\in\B^{qnil}$.
Since $\T\colon \A\to \B$ has the lifting property, there is $p\in \A^\bullet$ such that $\T (p)=q$. \par

Now, the identity $q\T (a)=\T (a)q$ implies that $pa - ap\in
\T^{-1}(0)$. However, multiplying by $1-p$, it is easy to prove that
$pa(1-p)$ and $(1-p)ap\in \T^{-1}(0)$. In addition, since $\T
(a+p)\in \B^{-1}$, $a+p\in \F_\T (\A)$. Finally, since $\T (pap)\in
\B^{qnil}$, $pap\in\mathcal{R}_\T (\A)$.

\indent Suppose that there exists $p\in \A^\bullet$ such that
$a+p\in \F_\T (\A)$, $pa(1-p)$ and $(1-p)ap\in \T^{-1}(0)$ and
$pap\in \mathcal{R}_\T (\A)$. Consequently, $q=\T (p)\in\B^\bullet$
and $q\T (a)=\T (a)q$, $\T (a)+q\in \B^{-1}$ and $\T
(a)q=q\T(a)q\in\B^{qnil}$. Thus, according to \cite[Theorem
4.2]{Koliha}, $\T (a)\in\B^\KD$, equivalently, $a\in
\mathcal{GBF}_\T (\A)$.\par

\noindent (ii). Apply the same argument used in the proof of
statement (i), using in particular \cite[Proposition 1(a)]{RS}
instead of \cite[Theorem 4.2]{Koliha}.
\end{proof}

\indent Next some basic properties of the objects introduced in Definition \ref{def3} will be considered.\par

\begin{teor} \label{inkluzije}Let $\A$ and $\B$ be two  unital Banach algebras and consider a (not necessarily continuous)
homomorphism $\T\colon \A\to\B$. Then, the following statements
hold.\par

\noindent {\rm (i)} $\T^{-1}(0) \subseteq
\T^{-1}(\B^\bullet)\subseteq \B\F_\T(\A)$.\par

\noindent {\rm (ii)} $\A^\bullet\subseteq
\T^{-1}(\B^\bullet)\subseteq \B\F_\T(\A)$.\par

\noindent {\rm (iii)}$\F_\T(\A)$ is a proper subset of
$\B\F_\T(\A)$.\par

\noindent {\rm (iv)} $\W_\T(\A)$ is a proper subset of
$\B\W_\T(\A)$.\par

\noindent {\rm (v)} $\B_\T(\A)$ is a proper subset of $\B\B_\T
(\A)$.\par

\noindent {\rm (vi)} $\A^\bullet\setminus
 \T^{-1}(1)\subseteq\B\F_\T (\A)\setminus\F_\T (\A)$.\par

\noindent {\rm (vii)}If $a, \; b \in \B\F_\T (\A)$ are such that
$ab-ba \in \T^{-1}(0)$, then $ab \in \B\F_\T (\A)$.\par

\noindent {\rm (viii)} If $a \in \B\W_\T (\A)$, then $a^n \in
\B\W_\T (\A)$ for every $n \in \NN$.\par

\noindent {\rm (ix)} If $a \in \B\B_\T (\A)$, then $a^n \in \B\B_\T
(\A)$ for every $n \in \NN$.\par

\noindent {\rm (x)} $\B\W_{\T} (\A)\setminus \W_{\T} (\A)\subseteq
\B\F_{\T} (\A)\setminus \F_{\T} (\A)$.\par

\noindent {\rm (xi)}  $\B\B_{\T}(\A)\setminus \B_{\T}(\A)\subseteq
\B\F_{\T}(\A)\setminus \F_{\T} (\A)$.\par

\noindent {\rm (xii)} $\A^\KD \cap \B\F^{(1)}_{\T} (\A)\subseteq
\B\B^{(1)}_{\T} (\A)$.\par

\noindent {\rm (xiii)}  $\sigma_{\B\W_\T}(a)=\bigcap _{c\in
\T^{-1}(0)}\sigma_{\D}(a+c)$ ($a\in\A$).\par

\noindent {\rm (xiv)}  $\sigma_{\B\B_\T}(a)=\bigcap_{c \in
\T^{-1}(0), ac=ca} \sigma_\D(a+c)$ ($a\in\A$).\par

\noindent {\rm (xv)} The sets $\sigma_{\B\W_\T}(a)$ and
$\sigma_{\B\B_\T}(a)$ are closed ($a\in\A$).
\end{teor}

\begin{proof}(i). This statement can be easily derived from the inclusions
$$
\{0\}\subseteq \B^\bullet\subseteq\B^\D.
$$

\noindent (ii). Clearly, $\T (\A^\bullet)\subseteq \B^\bullet$ and
$\B^\bullet\subseteq \B^\D$.

\noindent (iii).  Since  $\{0\}\cap\B^{-1}=\emptyset$, then
$\T^{-1}(0)\cap\F_\T(\A)=\T^{-1}(0)\cap T^{-1}(\B^{-1})= \emptyset$.
Consequently, $\T^{-1}(0)\subseteq \B\F_\T(\A)\setminus \F_\T
(\A)$.\par

\noindent (iv). Clearly, $\T^{-1}(0) \subseteq\B\W_\T (\A)$. In
addition, according to the proof of statement (iii),
$\T^{-1}(0)\cap\W_\T(\A)\subseteq
\T^{-1}(0)\cap\F_\T(\A)=\emptyset$. Therefore, $\T^{-1}(0)\subseteq
\B\W_\T(\A)\setminus \W_\T  (\A)$.\par

\noindent (v). It is clear that $\T^{-1}(0) \subseteq \B\B_\T (\A)
\setminus \B_\T (\A)$.\par

\noindent (vi). Note that $\A^\bullet\setminus\T^{-1}(1)\subseteq
\A^\bullet\subseteq \B\F_\T (\A)$. In addition, if $a\in
\A^\bullet\setminus\T^{-1}(1)$, then $\T (a)\in\B^\bullet\setminus
\B^{-1}$. In particular, $a\notin \F_\T(\A)$.\par

\noindent (vii). Apply  \cite[Proposition 2.6]{Ber1}.

\noindent (viii). Let $a \in \B\W_\T (\A)$. Then $a=b+c$, where $b
\in \A^\D$ and $c \in \T^{-1}(0)$. It will be proved that
$a^n=b^n+x_n$, where $x_n \in \T^{-1}(0)$, for every $n \in \NN$. In
fact, for $n=1$ it is obvious. Suppose that this statement  is true
for $k \in \NN$. Then,
\[a^{k+1}=a^k a=(b^k+x_k)(b+c)=b^{k+1}+(b^k c+x_k b+x_k c).\]

Clearly,  $b^k c+x_k b+x_k c \in \T^{-1}(0)$. As a result, since
$b^n\in \A^\D$, $a^n \in \B\W_\T (\A)$, $n\in\NN$.

\noindent (ix). Let $a \in \B\B_\T(\A)$. In particular, there are
$b\in \A^\D$ and $c\in \T^{-1}(0)$ such that  $a=b+c$ and $bc=cb$.
As a result, for every $n \in \NN$
\[a^n=(b+c)^n=\sum_{k=0}^{n} \binom{n}{k} b^{n-k}c^k=b^n+ \sum_{k=1}^{n} \binom{n}{k} b^{n-k}c^k.\]
Since $b^n\in \A^\D$, $z=\sum_{k=1}^{n} \binom{n}{k} b^{n-k}c^k\in
\T^{-1}(0)$ and $b^n$ commutes with $z$,  $a^n \in \B\B_\T(\A)$.\par

\noindent (x). Clearly,  $\B\W_{\T} (\A)\setminus
\W_{\T}(\A)\subseteq \B\W_{\T}(\A)\subseteq \B\F_{\T}(\A)$. If
 $a \in \B\W_\T (\A)\setminus \W_\T (\A)$, then there exist $c \in\A^\D$ and
$d \in \T^{-1}(0)$ such that $a=c+d$. In addition, according to
\cite[Proposition 1(a)]{RS}, there is $p\in \A^\bullet$ such that
\[cp=pc, \; \; \; \; \; \; c+p \in \A^{-1}, \; \; \; \; \; \; cp \hbox{ is nilpotent}.
\]
Note that since  $a=(c+p)+(d-p)$ and $a\notin W_\T (\A)=\A^{-1}+
\T^{-1}(0)$, $p\notin \T^{-1}(0)$. Let $0\neq q=\T (p)\in
\B^\bullet$. Then, $q\T (c)=\T (c)q$, $\T (c)+q\in \B^{-1}$ and $\T
(c)q$ is nilpotent. Thus, $\T (c)$ is Drazin invertible but not
invertible ($q\neq 0$), which implies that $c\notin \F_{\T}(\A)$.
However, since $d\in \T^{-1}(0)$, $a\notin \F_{\T}(\A)$.\par
\noindent (xi). Apply an argument similar to the one used in the
proof of statement (x).\par

\noindent (xii). Since $\A^{-1}\subseteq \B\B_\T^{(1)}(\A)$, suppose
that $a\in \A^\KD\setminus \A^{-1}$. According to  \cite[Theorem
6.4]{Koliha}, there is $p\in\A^\bullet$ such that $pa=ap$,
$a+p\in\A^{-1}$ and $a= x+y$, where $y=ap\in\A^{qnil}$, $x=a(1-p)$,
$xy=yx=0$ and $x\in\A^\sharp$. In addition, since $\T (a)\in
\B^\sharp$, there is $q\in \B^\bullet$ such that $q\T (a)=\T (a)q=0$
and $\T (a)+q\in\B^{-1}$ (\cite[Lemma 3]{RS}). Now well, if $q=0$, then by \cite[Theorem
2.4]{MR} $a\in\F_\T (\A)\cap\A^{\KD}\subseteq\B_\T(\A)\subseteq
\B\B^{(1)}_\T (\A)$. On the other hand, if $q\neq 0$, then according
to \cite[Theorem 3.1]{Koliha}, $\T (p)=q$ and $\T (y)=\T (a) \T
(p)=0$. Thus, $a\in \B\B^{(1)}_\T (\A)$.\par

\noindent (xiii)-(xiv). These statements can be easily deduced.\par

\noindent (xv). Apply statements (xiii)-(xiv).
\end{proof}
\noindent The following theorem summarizes the basic properties of
generalized B-Fredholm elements.

\begin{teor} \label{4444}Let $\A$ and $\B$ be two  unital Banach algebras and consider a  (not necessarily continuous)
homomorphism $\T\colon \A\to\B$. Then, the following statements
hold.\par \noindent {\rm (i)} $\T^{-1}(\B^{qnil} \setminus \B^{nil})
\subseteq \mathcal{GBF}_\T(\A) \setminus \mathcal{BF}_\T(\A)$.

\noindent {\rm (ii)}If $a, \; b \in \mathcal{GBF}_\T (\A)$ are such
that $ab-ba \in \T^{-1}(0)$, then $ab \in \mathcal{GBF}_\T
(\A)$.\par

\noindent {\rm (iii)} If $a \in \mathcal{GBW}_\T (\A)$, then $a^n
\in \mathcal{GBW}_\T (\A)$ for every $n \in \NN$.\par

\noindent {\rm (iv)} If $a \in \mathcal{GBB}_\T (\A)$, then $a^n \in
\mathcal{GBB}_\T (\A)$ for every $n \in \NN$.\par

\noindent {\rm (v)} $\mathcal{GBW}_{\T} (\A)\setminus \W_{\T}
(\A)\subseteq \mathcal{GBF}_{\T} (\A)\setminus \F_{\T} (\A)$.\par

\noindent {\rm (vi)}  $\mathcal{GBB}_{\T}(\A)\setminus
\B_{\T}(\A)\subseteq \mathcal{GBF}_{\T}(\A)\setminus \F_{\T}
(\A)$.\par

\noindent {\rm (vii)}  $\sigma_{\mathcal{GBW}_\T}(a)=\bigcap _{c\in
\T^{-1}(0)}\sigma_{\KD}(a+c)$ ($a\in\A$).\par

\noindent {\rm (viii)}  $\sigma_{\mathcal{GBB}_\T}(a)=\bigcap_{c \in
\T^{-1}(0), ac=ca} \sigma_\KD(a+c)$ ($a\in\A$).\par

\noindent {\rm (ix)} The sets $\sigma_{\mathcal{GBW}_\T}(a)$ and
$\sigma_{\mathcal{GBB}_\T}(a)$ are closed ($a\in\A$).
\end{teor}
\begin{proof} \noindent (i). It follows from the fact that $\B^{qnil} \setminus \B^{nil}
\subseteq \B^\KD \setminus \B^\D$. \par
\noindent (ii). Apply \cite[Theorem 5.5]{Koliha}.\par
\noindent (iii)-(iv). These statements can be proved using an argument similar to the one in
Theorem \ref{inkluzije}(viii) and Theorem \ref{inkluzije}(ix), respectively.\par
\noindent (v)-(vi). Apply an argument similar to the one in Theorem \ref{inkluzije}(x)-(xi),
using \cite[Theorem 3.1]{Koliha} instead \cite[Proposition 1(a)]{RS}.\par
\noindent (vii)-(viii). These statements can be easily deduced.\par
\noindent (ix). Apply statements (vii) and (viii).
\end{proof}

\section{Perturbations of B-Fredholm elements}

\indent To prove the main results of this section, some preparation
is needed.\par

\indent Given  $S \subseteq \A$, \it the commuting perturbation class of $S$, \rm
is the set
\[P_{comm}(S)=\{a \in \A: S +_{comm} \{a\} \subset S\}, \]
where, if $H,K \subseteq \A$
\[H+_{comm} K=\{c+d: (c,d) \in H \times K, cd=dc\}.\]

\bigskip

\indent Let $\A$ and $\B$ be two unital Banach algebras and consider
the homomorphism $\T\colon \A \to \B$. Given $K\subseteq \B$, it is not difficult to prove that
\begin{equation}\label{gen}
\T^{-1}(P_{comm}(K))\subseteq P_{comm}(\T^{-1}(K)).
\end{equation}
In particular, $\T^{-1}(P_{comm}(\B^\D))\subseteq
P_{comm}(\T^{-1}(\B^\D))=P_{comm} (\B\F_\T (\A))$. Moreover, it is
clear that $\T^{-1} (\Poly^{-1}(\{0\})) = \Poly^{-1}(\T^{-1}(0))$.
Also, if $K_1, \, K_2 \subseteq \A$ are such that $K_1 \cap K_2 \neq
\emptyset$ then
\begin{equation}\label{presek}
P_{comm}(K_1) \cap P_{comm}(K_2) \subseteq P_{comm}(K_1 \cap K_2).
\end{equation}

\indent On the other hand, it is well known that $b \in \A^{qnil}$ if and only if for
every $a \in \A$ which commutes with $b$ there is the equivalence:
\begin{equation}\label{eekv1}
a \in \A^{-1} \Longleftrightarrow a+b \in \A^{-1},
\end{equation}
that is,
\begin{equation}\label{eekv2}
a \notin \A^{-1} \Longleftrightarrow a+b \notin \A^{-1}.
\end{equation}
The equivalences \eqref{eekv1} and \eqref{eekv2} hold also if
$\A^{-1}$ is replaced by $\A^{-1}_{\left}$ or $\A^{-1}_{\right}$.
Consequently,
\begin{equation*}
\A^{qnil}=P_{comm}( \A^{-1})=P_{comm}( \A^{-1}_{\left})=P_{comm}(
\A^{-1}_{\right}),
\end{equation*}
 and also,
\begin{equation}\label{complement}
\A^{qnil}=P_{comm}(\A \setminus \A^{-1})=P_{comm}(\A \setminus
\A^{-1}_{\left})=P_{comm}(\A \setminus \A^{-1}_{\right}).
\end{equation}

\indent In first place, a preliminary result is considered.\par

\begin{prop} \label{algebraic}
Let $\A$ be a unital Banach algebra and consider an algeraic element $a \in \A$. Then,
$a\in\A^{qnil}$ if and only if $a$ is nilpotent.
\end{prop}
\begin{proof}
Every nilpotent element is quasinilpotent. On the other hand, if $a\in\A^{qnil}$, then let $P\in \CC [X]$ be the minimal
polynomial such that $P(a)=0$. It is well known that
$\sigma(a)=P^{-1}(\{0\})$. Since $\sigma(a)=\{0\}$, there must exist
$k\in\NN$ such that $P(X)=X^k$. Consequently, $a$ is nilpotent.
\end{proof}

\indent In the following theorem the commuting perturbation class of
$\A^\D$ and $\A^\KD$ will be considered.\par

\begin{teor} \label{inkluzije1}
Let $\A$ be a unital Banach algebra. Then,
\par

\noindent {\rm (i)} $\A^{nil} \subseteq P_{comm}(\A^\D) \subseteq
\Poly^{-1}(\{0\}).$ \par

\noindent {\rm (ii)} $\A^{qnil} \subseteq P_{comm}(\A^{\KD}).$
\par

\noindent {\rm (iii)} $\A^{nil} \subseteq P_{comm}(\A^\D \setminus
\A^{-1}).$ \par

\noindent {\rm (iv)} $\A^{qnil} \subseteq P_{comm}(\A^\KD
\setminus \A^{-1}).$ \par
\end{teor}
\begin{proof}
{\rm (i)}. Let $b \in \A^{nil}$ and $a \in \A^\D$ such that
$ab=ba$. Since $b\in \A^\D$  and $b^d=0$, according to
\cite[Theorem 3]{Dragana}, $a+b \in \A^\D$. In order to prove
the remaining inclusion, suppose that $b \in P_{comm}(\A^\D)$. The
elements $\lambda 1 (=\lambda)$ are Drazin invertible and commute
with $b$ for every $\lambda \in \CC$. Therefore,
$b+\lambda\in\A^\D$ for every $\lambda \in \CC$. In particular,
$\sigma_{\D}(b)=\emptyset$. Therefore, according to \cite[Theorem 2.1]{Boasso},
$b$ is algebraic. \par

\noindent {\rm (ii)}. It follows from \cite[Theorem 8]{Dragana} and
from the fact that $b^D=0$ if $b \in \A^{qnil}$. \par

\noindent {\rm (iii)}. Since $\A^{nil} \subseteq P_{comm}(\A^\D)$
and $\A^{nil} \subseteq \A^{qnil}=P_{comm}(\A
\setminus \A^{-1})$ (identity  \eqref{complement}), apply \eqref{presek}
to obtain $\A^{nil} \subseteq P_{comm}(\A^\D \cap (\A \setminus
\A^{-1}))=P_{comm}(\A^\D \setminus \A^{-1}).$
\par

\noindent {\rm (iv)}. It follows from {\rm (ii)}, \eqref{complement}
and \eqref{presek}.
\end{proof}

\begin{corollary}\label{druga}
Let $\A$ be a unital Banach algebra. Then
\begin{equation}\label{druga-1}
\Poly^{-1}(\{0\})\cap \A^{qnil}\subseteq P_{comm}(\A^\D).
\end{equation}
\end{corollary}
\begin{proof}
Apply Proposition \ref{algebraic} and Theorem \ref{inkluzije1}
(i).
\end{proof}

\begin{rem}\label{remark444}\rm Let $\A$ be a unital Banach algebra.Then, the following statements hold.\par
\noindent (i) $\A^{\K\D}\setminus\A^{-1}_{\left}=\A^{\K\D}\setminus\A^{-1}_{\right}=\A^{\K\D}\setminus\A^{-1}$.\par
\noindent (ii) $\A^\D\setminus\A^{-1}_{\left}=\A^\D\setminus\A^{-1}_{\right}=\A^\D\setminus\A^{-1}$.\par
\noindent (iii) $\A^{\K\D}\cap\A^{-1}_{\left}=\A^{\K\D}\cap\A^{-1}_{\right}=\A^{-1}$.\par
\noindent (iv) $\A^\D \cap\A^{-1}_{\left}=\A^\D\cap\A^{-1}_{\right}=\A^{-1}$.\par

\indent To prove statement (i), note that $\A^{\K\D} \setminus\A^{-1}_{\left}\subseteq \A^{\K\D}
\setminus\A^{-1}$ and $\A^{\K\D}
\setminus\A^{-1}_{\right}\subseteq \A^{\K\D} \setminus\A^{-1}$. Now suppose that $a\in \A^{\K\D}\setminus
\A^{-1}$. Then $0\in\iso\, \sigma(a)$ (\cite[Theorem
4.2]{Koliha}). It follows that $0$ is a boundary point of $\sigma
(a)$ and hence, it belongs to the left and also to the right
spectrum of $a$. This implies that $a\notin\A^{-1}_{\left}$ and
$a\notin\A^{-1}_{\right}$. In particular, $a\in \A^{\K\D}
\setminus\A^{-1}_{\left}\cap \A^{\K\D}
\setminus\A^{-1}_{\right}$. Therefore, $\A^{\K\D}\setminus
\A^{-1}\subseteq \A^{\K\D} \setminus\A^{-1}_{\left}\cap  \A^{\K\D}
\setminus\A^{-1}_{\right}$.\par
\indent Statement (ii) can be proved using an argument simimilar to the one developed in
the previous paragraph.\par
\indent Statement (iii) (respectively statement (iv)) can be easily derived from
statement (i) (respectively statement (ii)).
\end{rem}

\indent Next algebraic (nilpotent) elements will be characterized using the Drazin spectrum.\par

\begin{teor}\label{ekvi} Let $\A$  be a unital Banach algebra and consider $d\in \A^{qnil}$. Then the following statements are
equivalent:\par
\noindent {\rm (i)} The element $d$ is algebraic.\par
\noindent {\rm (ii)} Given $a\in\A$,  $ad=da$ implies that $\sigma_\D(a+d)=\sigma_\D(a)$.
\end{teor}
\begin{proof} If $d$ is algebraic, then according to Proposition  \ref{algebraic}, $d\in \A^{nil}$, which is equivalent to
$-d\in \A^{nil}$. According to Corollary \ref{druga}, $d,-d\in
P_{comm}(\A^\D)$. Let $a\in \A$ such that $ad=da$. If $\lambda\in
\CC$ is such that $\lambda\notin\sigma_\D (a)$, then
$a-\lambda\in\A^\D$, and since $d\in P_{comm}(\A^\D)$,
$a+d-\lambda\in\A^\D$. In particular, $\lambda\notin\sigma_\D(a+d)$.
To prove the reverse, apply the same argument to $-d\in
P_{comm}(\A^\D)$, $a+d$ and $\lambda\notin\sigma_\D (a+d)$.\par

\indent Conversely, if $a=0$, then
$\sigma_\D(d)=\sigma_\D(0)=\emptyset$. However, according to
\cite[Theorem 2.1]{Boasso}, $d$ is algebraic.
\end{proof}

\indent In the following theorem the commuting perturbation class of (generalized)
B-Fredholm elements will be considered.\par

\begin{teor}Let $\A$ and $\B$ be two  unital Banach algebras and consider
a (not necessarily continuous) homomorphism  $\T: \A \to \B$. Then,
\par
\noindent  {\rm (i)} $\N_\T (\A) \subseteq
P_{comm}(\B\F_\T (\A))\subseteq T^{-1} (\Poly^{-1}(\{0\}).$ \par

\noindent  {\rm (ii)} $\R_\T (\A) \subseteq P_{comm}(\mathcal{GBF}_\T
(\A)).$
\par

\noindent  {\rm (iii)} $\N_\T (\A) \subseteq P_{comm}(\B\F_\T (\A)
\setminus \F_\T(\A)).$
\par

\noindent  {\rm (iv)} $\R_\T (\A) \subseteq P_{comm}(\mathcal{GBF}_\T
(\A) \setminus \F_\T(\A)).$
\par
\end{teor}
\begin{proof} {\rm (i)}. According to Theorem \ref{inkluzije1} {\rm (i)} and \eqref{gen},
$$
\N_\T (\A)\subseteq T^{-1}(P_{comm}(\B^\D))\subseteq
P_{comm}(T^{-1}(\B^\D))= P_{comm}(\B\F_\T (\A)).
$$
Let $a\in P_{comm}(\B\F_\T (\A))$. Then for every $\lambda\in\CC$,
$a+\lambda\in \B\F_\T (\A)$, equivalently, $T(a)+\lambda \in\B^\D$.
However, according to \cite[Theorem 2.1]{Boasso}, $T(a)\in\B$ is algebraic.\par

\noindent {\rm (ii)}-{\rm (iv)}. Apply Theorem \ref{inkluzije1}(ii)-(iv) and use an argument similar to the one
in the proof of statement (i).
\end{proof}

\begin{corollary}  \label{pposledica}Let $\A$ and $\B$ be two  unital Banach algebras and consider
a (not necessarily continuous) homomorphism  $\T: \A \to \B$. Let $a \in \B\F_\T (\A)$ and $b
\in \N_\T (\A)$ such that $ab-ba \in \T^{-1}(0)$. Then,
$a+b\in\B\F_\T (\A)$. In particular,
$\sigma_{\B\F_\T}(a+b)=\sigma_{\B\F_\T}(a)$.
\end{corollary}
\begin{proof}
Apply Theorem \ref{inkluzije1} (i).
\end{proof}

\begin{corollary}\label{prethodna} Let $\A$ and $\B$ be two  unital Banach algebras and consider
a (not necessarily continuous) homomorphism $\T: \A \to \B$. If $a\in\R_\T (\A)$ and $\T (a)\in\B$ is algebraic, then $a\in
P_{comm}(\B\F_\T (\A))$.
\end{corollary}
\begin{proof} According to Corollary \ref{druga},
$\T (a)\in P_{comm} (\B^\D)$. Therefore, $a\in
\T^{-1}($ $P_{comm}(\B^\D))\subseteq P_{comm}(\T^{-1}(\B^\D))=P_{comm}
(\B\F_\T (\A))$.
\end{proof}

\indent Under the same assumptions in Corollary \ref{prethodna}, note that if $a\in \A$ is algebraic and
$a\in\R_\T (\A)$, then $\T (a)\in P_{comm} (\B^\D)$.

\indent Let $\A$ and $\B$ be two  unital Banach algebras and
consider a (not necessarily continuous) homomorphism  $\T: \A \to \B$. Recall that according to
\cite[Theorem 10.1]{ZDH}, the following statements are
equivalent:\par \noindent (i) The element $d\in\R_\T (\A)$.\par
\noindent (ii) If $a\in \A$ is such that $ad-da\in \T^{-1}(0)$, then
$\sigma_{\F_\T}(a)=\sigma_{\F_\T}(a+d)$.\par \noindent (iii) If
$a\in \A$ is such that $ad=ad$, then
$\sigma_{\F_\T}(a)=\sigma_{\F_\T}(a+d)$.\par \noindent (iv)
$\sigma_{\F_\T}(d)=\{0\}$.

 \indent In the following theorem, a similar result for the B-Fredholm spectrum will be considered.\par

\begin{teor}Let $\A$ and $\B$ be two  unital Banach algebras and consider
a (not necessarily continuous) homomorphism  $\T: \A \to \B$. Let $d\in\R_\T (\A)$. Then, the
following conditions are equivalent.\par \noindent {\rm (i)} $\T(d)$
is algebraic.\par \noindent {\rm (ii)} If $a\in \A$ is such that
$ad-da\in \T^{-1}(0)$, then
$\sigma_{\B\F_\T}(a+d)=\sigma_{\B\F_\T}(a))$.\par \noindent {\rm
(iii)}  If $a\in \A$ is such that $ad=da$, then
$\sigma_{\B\F_\T}(a+d)=\sigma_{\B\F_\T}(a))$.\par \noindent {\rm
(iv)} $\sigma_{\B\F_\T}(d)=\emptyset$.
\end{teor}
\begin{proof} (i)$\Longrightarrow$(ii). Apply Theorem \ref{ekvi}.\par
\noindent (ii)$\Longrightarrow$(iii). It is obvious.\par
\noindent (iii)$\Longrightarrow$(iv). Consider $a=0$. Then,
$\sigma_{\B\F}(d)=\sigma_{\B\F}(0)=\emptyset$.\par
\noindent (iv)$\Longrightarrow$(i). Apply Theorem \ref{teo2.2}(iv).
\end{proof}

\section{Perturbations of (generalized) B-Fredholm elements with equal spectral idempotents}

\noindent To prove the main results of this section, some prepartion is needed first.\par

\indent  Let $\A$ be a  unital Banach algebra and consider $a\in
\A^\KD$. Let $p=1-a^Da$. If $a\in\A^{-1}$, then $p=0$, but if
$0\in $ iso $\sigma (a)$, $p$ is the \it spectral idempotent \rm
of $a$ corresponding to $0$ and it will be written $p=a^\pi$. Note
that $ap=pa$, $ap\in\A^{qnil}$, \  $a(1-p), a^D \in
((1-p)\A(1-p))^{-1}$ and $a^D$ is the inverse of $(1-p)a$ in the
algebra $(1-p)\A(1-p)$ (\cite{R2}).\par

\begin{rem}\label{rem5.1}\rm Let $\A$ and $\B$ be two unital Banach algebras and consider a (non necessarily continuous) homomorphism $\T\colon\A\to\B$.\par
\noindent (a). Let $a\in\A$ such that $\T (a)\in \B^\KD$. Let $\T
(a)^\pi=q$ and suppose that there exist $p\in\A^\bullet$ such that
$\T (p)=q$ and $w\in (1-p)\A (1-p)$ such that $\T (w)=\T (a)^D=
((1-q)\T(a) (1-q))^{-1}\in ((1-q)\B (1-q))^{-1}$ ( \cite[Theorem
4.2]{Koliha}). Then, it is not difficult to prove the following
statements.\par \noindent (i) $(1-p)aw=1-p +c_1$ and $wa(1-p)=1-p
+c_2$, where $c_i\in \T^{-1}(0)\cap (1-p)\A (1-p)$, $i=1, 2$.\par
\noindent (ii) If $w'\in (1-p)\A(1-p)$ is such that $\T (w')=\T
(a)^D$, then $w'-w\in \T^{-1}(0)\cap (1-p)\A (1-p)$.\par \noindent
(b). Suppose in addition that   $\T\colon\A\to\B$ is surjective and
has the lifting property and consider $a\in \A$ as before, i.e., $\T
(a)\in \B^\KD$ and  $\T (a)^\pi=q$ . In particular, there exist
$p\in\A^\bullet$ such that $\T (p)=q$ and $z\in\A$ such that
$\T(a)^D=\T(z)$. However, since  $\T(a)^D\in (1-q)\B(1-q)$, it is
possible to choose $z\in (1-p)\A(1-p)$.
\end{rem}

\indent The results of Remark \ref{rem5.1} will be used in what follows.

\begin{prop}\label{prop5.2} Let $\A$ and $\B$ be two unital Banach algebras and consider a (non necessarily continuous) homomorphism $\T\colon\A\to\B$.
 Let $a_1\in\A$ such that $\T (a_1)\in\B^\KD$ and  $\T (a_1)^\pi=q$. Suppose that there exist $p\in\A^\bullet$ and $w_1\in (1-p)\A(1-p)$
such that $\T (p)=q$ and $\T (w_1)=\T (a_1)^D$. Let  $a_2\in\A$  and
define $z= 1+ \T(a_1)^D\T(a_2-a_1)$. Then, the following statements
hold.\par \noindent {\rm (i)} The element $z\in\B^{-1}$ if and only
if $p+w_1a_2\in\F_\T (\A)$.\par \noindent {\rm (ii)} Suppose that
$\T (a_2)\T (a_1)^\pi=\T (a_1)^\pi \T (a_2)$. Then, $z\in  \B^{-1}$
if and only $p+w_1a_2(1-p)\in\F_\T (\A)$.\par
\end{prop}
\begin{proof}(i). Note that  $z\in \B^{-1}$ if and only if $1+ \T(w_1(a_2-a_1))\in\B^{-1}$.
Since $\T (w_1a_1)=\T(w_1a_1(1-p))=1-q$, necessary and sufficient for $z\in \B^{-1}$ is that $q+\T (w_1a_2)\in \B^{-1}$,
which in turn is equivalent to $p+w_1a_2\in\F_\T (\A)$.\par

\noindent (ii). Since $\T (a_1)$ and $\T (a_2)$ commute with $q$,
$z\in \B^{-1}$ if and only if $1-q + \T (w_1(a_2-a_1)(1-p))\in
((1-q)\B(1-q))^{-1}$. Now, using an argument similar to one in the
proof of statement (i), it is not difficult to prove that $z\in
\B^{-1}$ if and only if $p+w_1a_2(1-p)\in\F_\T (\A)$.
\end{proof}

\indent In the following theorems, (generalized) B-Fredholm elements that have the same spectral idempotents
relative to the homomorphism $\T$ will be characterized.\par

\begin{teor}\label{teo5.3}Let $\A$ and $\B$ be two unital Banach algebras and consider a (non necessarily continuous)  homomorphism $\T\colon\A\to\B$.
Suppose in addition that $\T\colon \A\to \B$ is surjective and has
the lifting property. Let $a_1\in\mathcal{GBF}_\T (\A)$ and consider
$p\in\A^\bullet$ such that $\T (p)=\T (a_1)^\pi$. Then, the
following statements are equivalent.\par \noindent {\rm (i)}
$a_2\in\mathcal{GBF}_\T (\A)$ and $\T (a_1)^\pi=\T (a_2)^\pi$.\par
\noindent {\rm (ii)}  $pa_{2} (1-p)$ and $(1-p)a_2p\in \T^{-1}(0)$,
$pa_2p\in \R_\T (\A)$ and $p+a_2\in \F_\T (\A)$.\par \noindent {\rm
(iii)} $pa_{2} (1-p)$ and $(1-p)a_2p\in \T^{-1}(0)$, $pa_2p\in \R_\T
(\A)$ and $p+w_1a_2(1-p)\in \F_\T (\A)$, where $w_1\in (1-p)\A(1-p)$
is such that $\T (w_1)=\T (a_1)^D$.\par \noindent {\rm (iv)}
$a_2\in \mathcal{GBF}_\T (\A)$, $p+w_1a_2\in \F_\T (\A)$ and
$w_1=(p+w_1a_2)w_2 +c$, where $w_1$ is as in statement {\rm (iii)},
$w_2\in\A$ is such that $\T (w_2)=\T (a_2)^D$ and $c\in
\T^{-1}(0)$.\par
\end{teor}
\begin{proof}Apply Proposition \ref{prop5.2} to \cite[Theorem 2.2]{R2} and note the following fact (statement (iv)).\par

\indent According to the proof of Proposition \ref{prop5.2}(i), the identity
$$
\T( a_2)^D=(1+\T (a_1)^D\T(a_2-a_1))^{-1}\T (a_1)^D
$$
is equivalent to
$$
(q+ \T (w_1a_2))\T (w_2)=\T (w_1),
$$
which in turn is equivalent to $w_1=(p+w_1a_2)w_2 +c$, $c\in
\T^{-1}(0)$.
\end{proof}

\begin{rem}\label{rem5.5}\rm Under the same hypothesis of Theorem \ref{teo5.3}, let $a_2\in\A$ be such that
$a_2\in \mathcal{GBF}_\T (\A)$ and $\T (a_1)^\pi=\T (a_2)^\pi$.\par
\noindent (i) In particular, $w_2\in (1-p)\A (1-p)$ and
$$
w_1=(p+w_1a_2)w_2 +c= w_1a_2w_2+c= w_1(1-p)a_2(1-p)w_2 +c,
$$
where $c\in \T^{-1}(0)\cap (1-p)\A (1-p)$.\par \noindent (ii)
Consider again the identity
$$
\T (a_2)^D=(1+ \T(a_1)^DT(a_2-a_1))^{-1}\T (a_1)^D.
$$
It is not difficult to prove that
$$
\T (w_2a_1(1-p))\T (w_1a_2(1-p))=\T (w_1a_2(1-p))\T (w_2a_1(1-p))=1-q.
$$
However,
since $z= q+\T (w_1a_2(1-p))$ (Proposition \ref{prop5.2}(ii)), $z^{-1}=q+\T (w_2a_1(1-p))$.\par
In particular,
$$
\T (a_2)^D=(1+ \T(a_1)^DT(a_2-a_1))^{-1}\T (a_1)^D
$$
is equivalent to $w_2=w_2a_1(1-p)w_1 +d$, $d\in \T^{-1}(0)\cap
(1-p)\A(1-p)$.
\end{rem}

\indent In the next theorem B-Fredholm elements that have the same spectral idempotents
relative to the homomorphism $\T$ will be characterized.\par

\begin{teor}\label{teo5.4}Let $\A$ and $\B$ be two unital Banach algebras and consider a (non necessarily continuous)  homomorphism $\T\colon\A\to\B$.
Suppose in addition that $\T\colon \A\to \B$ is surjective and has
the lifting property. Let $a_1\in\mathcal{BF}_\T (\A)$ and consider
$p\in\A^\bullet$ such that $\T (p)=\T (a_1)^\pi$. Then, the
following statements are equivalent.\par \noindent {\rm (i)} $a_2\in
\mathcal{BF}_\T (\A)$ and $\T (a_1)^\pi=\T (a_2)^\pi$.\par \noindent
{\rm (ii)}  $pa_{2} (1-p)$ and $(1-p)a_2p\in \T^{-1}(0)$, $pa_2p\in
\N_\T (\A)$ and $p+a_2\in \F_\T (\A)$.\par \noindent {\rm (iii)}
$pa_ {2}(1-p)$ and $(1-p)a_2p\in \T^{-1}(0)$, $pa_2p\in \N_\T (\A)$
and $p+w_1a_2(1-p)\in \F_\T (\A)$, where $w_1\in (1-p)\A(1-p)$ is
such that $\T (w_1)=\T (a_1)^d$.\par \noindent {\rm (iv)} $a_2\in
\mathcal{BF}_\T (\A)$, $p+w_1a_2\in \F_\T (\A)$ and
$w_1=(p+w_1a_2)w_2 +c$, where $w_1$ is as in statement {\rm (iii)},
$w_2\in\A$ is such that $\T (w_2)=\T (a_2)^d$ and $c\in
\T^{-1}(0)$.\par
\end{teor}
\begin{proof} The same arguments in Remark \ref{rem5.1}, Proposition \ref{prop5.2} and Theorem \ref{teo5.3} apply to the case
of B-Fredholm elements using nilpotent elements instead of quasi-nilpotent elements. What is more, when considering Drazin invertible
Banach algebra elements,  statements similar to the ones in \cite[Theorem 2.2]{R2} hold, if nilpotent elements instead of quasi-nilpotent elements are used.
\end{proof}

\noindent Recall that in Theorem \ref{4444}(ii) (respectively Theorem \ref{inkluzije}(vii)) conditions assuring that the product of two generalized B-Fredholm elements (respectively two B-Fredholm elements) is generalized B-Fredholm (respectively B-Fredholm) were given. In the following theorem more information concerning this problem will be given.\par

\begin{teor}\label{teo5.6}Let $\A$ and $\B$ be two unital Banach algebras and consider a (non necessarily continuous)  homomorphism $\T\colon\A\to\B$ such that
$\T$ is surjective and has the lifting property. Let
$a_i\in\mathcal{GBF}_\T (\A)$, $i=1, 2$, such that $\T (a_1)^\pi=\T
(a_2)^\pi=q$ and $a_1a_2-a_2a_1\in \T^{-1}(0)$. Let $p\in\A^\bullet$
such that $\T (p)=q$. Then, $a_1a_2\in   \mathcal{GBF}_\T (\A)$, $\T
(a_1a_2)^\pi=q$ and if $w_1, w_2$ and $w_{12}\in (1-p)\A(1-p)$ are
such that  $\T (w_1)=\T (a_1)^D$, $\T (w_2)=\T (a_2)^D$ and $\T
(w_{12})=\T (a_1a_2)^D$, then $w_{12}=w_2w_1+c$, $c\in \T^{-1}(0)$.
\end{teor}\label{teo5.4}
\begin{proof}According to Theorem \ref{4444}(ii), $a_1a_2 \in \mathcal{GBF}_\T(\A)$.
Moreover, since $\T(a_1), \T(a_2) \in \B^\KD$ and
$\T(a_1)\T(a_2)=\T(a_2)\T(a_1)$, according to \cite[Theorem 5.5]{Koliha},
$\T(a_1a_2)^D=\T(a_1)^D\T(a_2)^D=\T(a_2)^D\T(a_1)^D$.
Further, since
$\T(a_1)^\pi=\T(a_2)^\pi=q$,
\begin{eqnarray*}
\T(a_1a_2)^\pi&=&1-\T(a_1a_2)\T(a_1a_2)^D=1-\T(a_1)\T(a_2)\T(a_2)^D\T(a_1)^D\\&=&1-\T(a_1)(1-q)\T(a_1)^D=1-(1-q)\T(a_1)\T(a_1)^D\\&=&
1-(1-q)(1-q)=1-(1-q)\\
&=&q.\\
\end{eqnarray*}

\indent Since $\T(a_1a_2)^D=\T(a_2)^D\T(a_1)^D$, $\T(a_1a_2)^D=\T (w_2)\T
(w_1)$. Consequently, $w_{12}=w_2w_1+c$, $c\in
\T^{-1}(0)$.
\end{proof}

\bigskip

\noindent {\bf Acknowledgements.} The first and third authors
are supported by the Ministry of Education, Science and
Technological Development, Republic of Serbia, grant no. 174007.

\end{document}